\documentclass[12pt]{amsart}
\usepackage{amsmath, amsthm, amssymb}
\usepackage[headings]{fullpage}
\usepackage[pdfstartview={FitH}]{hyperref}

\usepackage{color}

 \numberwithin{equation}{section}
 
 \theoremstyle{plain}
 \newtheorem{theorem}[equation]{Theorem}
 \newtheorem{corollary}[equation]{Corollary}
 \newtheorem{lemma}[equation]{Lemma}
 
 \newtheorem{proposition}[equation]{Proposition}
 
 \newtheorem*{cohenthm}{Cohen's Theorem}
 
 \theoremstyle{definition}
 \newtheorem{definition}[equation]{Definition}
 
 \newtheorem{example}[equation]{Example}
  
 \newtheorem{remark}[equation]{Remark}

\DeclareMathOperator{\ann}{ann}
\DeclareMathOperator{\Spec}{Spec}
\DeclareMathOperator{\Max}{Max}
 
\newcommand{\F}{\mathcal{F}}
\newcommand{\C}{\mathcal{C}}
\newcommand{\setS}{\mathcal{S}}
\newcommand{\Z}{\mathbb{Z}}

\begin{document}

\title{A Prime Ideal Principle for two-sided ideals}
\author{Manuel L. Reyes}
\address{Department of Mathematics\\
Bowdoin College\\
8600 College Station\\
Brunswick, ME 04011--8486, USA}
\email{reyes@bowdoin.edu}
\urladdr{http://www.bowdoin.edu/~reyes/}

\thanks{This material is based upon work supported by the National Science Foundation 
under Grant No.\ DMS-1407152.}

\date{June 30, 2016}

\begin{abstract}
Many classical ring-theoretic results state that an ideal that is maximal with respect to satisfying
a special property must be prime. We present a ``Prime Ideal Principle'' that gives a uniform
method of proving such facts, generalizing the Prime Ideal Principle for commutative rings
due to T.\,Y.~Lam and the author.
Old and new ``maximal implies prime'' results are presented, with results touching on
annihilator ideals, polynomial identity rings, the Artin-Rees property, Dedekind-finite rings,
principal ideals generated by normal elements, strongly noetherian algebras, and just
infinite algebras.
\end{abstract}

\maketitle

\section{Introduction}

A well-known phenomenon in ring theory is that, for certain properties $\mathcal{P}$, \emph{ideals 
of a ring that are maximal with respect to satisfying $\mathcal{P}$ are prime.} This is probably most 
familiar in the setting of commutative algebra, as in~\cite[pp.~70--71, 84--85]{Eisenbud}. For 
instance, an ideal in any  commutative ring $R$ is prime if it is maximal with respect to any of the 
following properties:
\begin{itemize}
\item being a proper subset of $R$,
\item not being finitely generated,
\item not being principal, or
\item having empty intersection with a given multiplicative submonoid $S \subseteq R$.
\end{itemize}
The \emph{Prime Ideal Principle} of~\cite{LR} unifies the above results along with many other classical
``maximal implies prime'' results in commutative algebra. The basic idea behind this
unification was to find certain properties of a family $\F$ of ideals in a fixed commutative ring that 
guarantee that an ideal maximal with respect to not lying in $\F$ is prime. 
The most notable such condition is that the family $\F$ be an \emph{Oka family}; this means that,
for any ideal $I \unlhd R$ and any element $a \in R$, the following implication holds:
\[
(I,a),(I:a) \in \F \implies I \in \F.
\]
The Prime Ideal Principle then states that, for any Oka family $\F$ in a commutative ring $R$, 
an ideal $I \unlhd R$ maximal with respect to $I \notin \F$ is prime. All of the examples above
were recovered from the Prime Ideal Principle in~\cite{LR} by demonstrating that a relevant family 
of ideals is Oka. 

A version of the Prime Ideal Principle for right ideals in noncommutative rings was established
in~\cite{Reyes}. This was achieved by identifying an appropriate type of ``prime right
ideal'' (these were called \emph{completely prime right ideals}) and by finding a suitable definition
of Oka families of right ideals. 
However, to date there is no version of the Prime Ideal Principle for two-sided ideals in noncommutative
rings.

The purpose of this paper is to present such a theory for two-sided ideals in noncommutative 
rings, unifying many of the classical noncommutative ``maximal implies prime'' results and
presenting some new ones.
In Section~\ref{PIP section} we briefly introduce Oka families of (two-sided) ideals in noncommutative
rings and prove the corresponding Prime Ideal Principle: if $\F$ is an Oka family of ideals in a
noncommutative ring $R$ and $I$ is an ideal that is maximal with respect to $I \notin \F$, then
$I$ is prime.
We also introduce stronger conditions on families of ideals called $(P_1)$, $(P_2)$, and $(P_3)$
that are noncommutative analogues  of those defined in~\cite{LR}. 

The heart of the paper lies in the remaining sections, which present many examples of families 
satisfying these different properties with corresponding applications of the Prime Ideal Principle. 
Section~\ref{P_i family section} gives examples of families satisfying the properties 
$(P_1)$, $(P_2)$, and $(P_3)$. Section~\ref{Oka family section} considers families that satisfy 
the weaker $r$-Oka and Oka properties, many of which are constructed from various classes
of (bi)modules; it concludes with a discussion of problems for future work.

\textbf{Acknowledgments.} I wish to thank George Bergman, T.\,Y.~Lam, and Lance Small 
for helpful conversations, suggestions, and references to the literature. I also thank the
referee for offering several clarifying remarks.

\textbf{Conventions.} Throughout this paper, all rings, modules, and ring homomorphisms are
assumed to be unital.
Let $R$ be a ring. We write $I \unlhd R$ to denote that $I$ is an ideal of $R$.
If $X$ is a subset of $R$, we write $(X)$ for the ideal of $R$ generated by $X$;
in case $X = \{x\}$ is a singleton we simply write $(x) = (\{x\})$.
Furthermore, if $J \unlhd R$ and $a \in R$, then we denote $(I,J) = (I \cup J) = I+J$ and 
$(I,a) = (I \cup \{a\}) = I+(a)$.
The notation $M_R$ and ${}_R M$ respectively indicate that $M$ is a right or left $R$-module.

\section{The Prime Ideal Principle}
\label{PIP section}

The families of ideals defined below illustrate one possible way to
generalize those ideal families defined in~\cite{LR} from commutative rings to
noncommutative rings. For two ideals $I,J\unlhd R$ we denote $J^{-1}I = 
\{x \in R : Jx \subseteq I\} $ and $IJ^{-1} = \{ x\in R : xJ \subseteq I\}$, 
both of which are clearly ideals in $R$.

\begin{definition}
\label{def:Oka}
Let $R$ be a ring and let $\F$ be a family of ideals of $R$ with $R\in \F$. We say that
$\F$ is an \emph{Oka family of ideals} if, for any $I \unlhd R$ and $a \in R$, 
\begin{equation}
\label{Oka property}
(I,a),\ (a)^{-1}I,\ I(a)^{-1} \in \F \implies I \in \F.
\end{equation}
We say that $\F$ is an \emph{$r$-Oka family of ideals} if, for any
$I \unlhd R$ and $a \in R$, the condition $(I,a),(a)^{-1}I \in \F$ implies that
$I \in \F$.
Furthermore, we define the following properties of the family $\F$ (recall
the standing assumption that $R \in \F$), where $I, J \unlhd R$ are arbitrary ideals:
\begin{itemize}
\item \emph{Monoidal}: $I,J \in \F \implies IJ \in \F$;
\item \emph{Semifilter}: $I \in \F$ and $J \supseteq I \implies J \in \F$;
\item $(P_1)$: $\F$ is a monoidal semifilter;
\item $(P_2)$: $\F$ is monoidal, and for any $J \in \F$, $I^{2}\subseteq J\subseteq I \implies I\in \F$;
\item $(P_3)$: For $A,B \in \F$ and $I \unlhd R$, $AB \subseteq I \subseteq A \cap B \implies I \in \F$;
\item \emph{Strongly $r$-Oka}: $I+J,J^{-1}I \in \F \implies I \in \F$;
\item \emph{Strongly Oka}: $I+J,J^{-1}I,IJ^{-1} \in \F \implies I \in \F$.
\end{itemize}
\end{definition}

The $r$-Oka and strongly $r$-Oka familes defined above have obvious left-handed analogues,
which we refer to as (strongly) $\ell$-Oka families.
An alternative term would have been ``(strongly) right- and left-Oka'' families. However, we have
avoided this terminology since it would have suggested that an Oka family is the same as a family
that is both right and left Oka.
It seems likely that an Oka family need not be either $r$-Oka nor $\ell$-Oka in full generality (such 
as the families produced via Proposition~\ref{Tensor category}), but we do not have an explicit
example to verify this.

Each of the properties of ideal families listed above is easily seen to
pass to arbitrary intersections of such families. Thus for each of these
properties, the set of all ideal families in $R$ with that property forms
a complete sublattice of the lattice of \emph{all} families of ideals in $R$.

The following result outlines the logical dependence between the many types
of families defined above.

\begin{proposition}
\label{logical dependence}
For a family $\F$ of ideals in a ring $R$, we have the following
implications of the properties defined above:
\begin{equation*}
\begin{array}{cccc}
(P_1) \implies (P_2) \implies (P_3) \implies & \text{strongly $r$-Oka} & \implies & \text{$r$-Oka} \\ 
& \Downarrow &  & \Downarrow \\ 
& \text{strongly Oka} & \implies & \text{Oka}
\end{array}
\end{equation*}
\end{proposition}

\begin{proof}
The horizontal implications are proved in the commutative case in~\cite[Thm.~2.7]{LR}, and
those arguments apply here with virtually no adjustment. The vertical
implications are obvious.
\end{proof}

We mention briefly that there are various equivalent characterizations of the $(P_2)$ and
$(P_3)$ conditions given in~\cite[\S 2]{LR} and~\cite[\S 2]{LR2}. Though the statements and
proofs are given for commutative rings, they once again transfer easily to the
noncommutative case. Note that because ($P_2$) $\Rightarrow$ ($P_3$) and because any ($P_3$) 
family is monoidal, a family satisfying any property ($P_i$) is monoidal.

Before proving the Prime Ideal Principle in this setting, we establish a variant of the
usual criterion for a ring to be prime.
If $R$ is a prime ring and $a,b\in R \setminus \{0\}$ then one can conclude that both
$aRb\neq 0$ and $bRa\neq 0$. There is an inherent symmetry in the definition
of a prime ring in the sense that we can interchange the roles of $a$ and $b$.
What if $R$ fails to be prime? We know offhand that there exist
$a,b\in R\setminus \{0\}$ such that $aRb=0$. But if $R$
is not commutative, we cannot necessarily conclude that $bRa=0$. At first
glance, the symmetry seems to disappear in the \emph{failure} of a ring to
be prime. The following lemma rectifies this situation.

\begin{lemma}
\label{prime ring characterization}
For any ring $R\neq 0$, the following are equivalent:
\begin{itemize}
\item[\rm (1)] $R$ is a prime ring
\item[\rm (2)] For $a,b\in R\setminus\{ 0\}$, \emph{both} $aRb\neq 0$ and $bRa\neq 0$
\item[\rm (3)] For $a,b\in R\setminus\{ 0\}$, \emph{either} $aRb\neq 0$ or $bRa\neq 0$.
\end{itemize}
In particular if $R$ is not a prime ring, then there exist $0\neq a,b\in R$
such that $aRb=0=bRa$.
\end{lemma}

\begin{proof}
The implications $(1) \Rightarrow  (2) \Rightarrow (3)$ are immediate. To prove 
$(3) \Rightarrow (1)$ we will verify the contrapositive.
Suppose that $R$ is not a prime ring; we will produce $a, b \in R \setminus \{0\}$
such that $aRb = bRa = 0$.
Because the ring is not prime, there exist $x,y\in R\setminus \{0\}$
such that $xRy=0$. If $yRx=0$ then we are finished.
Otherwise there exists $r \in R$ such that $yrx \neq 0$. Setting $a = b = yrx$
yields $aRb = bRa = 0$.
\end{proof}

We arrive at the following noncommutative version of the Prime Ideal Principle
for two-sided ideals in noncommutative rings.
Given a family $\F$ of ideals in a ring $R$, its \emph{complement} is 
$\F' = \{I \unlhd R : I \notin \F\}$. We use $\Max(\F')$ to denote the set of ideals 
that are maximal in the complement of $\F$. Also, we let $\Spec(R)$ denote the
set of prime ideals of $R$.

\begin{theorem}[Prime Ideal Principle]
\label{PIP}
For any Oka family $\F$ of ideals in a ring $R$, we have $\Max(\F') \subseteq \Spec(R)$.
\end{theorem}

\begin{proof}
Let $I \in \Max(\F')$, and assume for contradiction that $I$ is not prime.
By Lemma~\ref{prime ring characterization} there exist $a,b \in R \setminus I$
such that $aRb,bRa \subseteq I$. Then $a \in (I,a) \setminus I$ and
$b \in (a)^{-1}I \cap I(a)^{-1} \setminus I$. Maximality of $I$ implies that
the ideals $(I,a)$, $(a)^{-1}I$, and $I(a)^{-1}$ lie in $\F$. Because $\F$ is Oka,
this means that $I \in \F$, a contradiction.
\end{proof}

The \emph{Prime Ideal Principle supplement} of~\cite{LR} was a streamlined way
to apply the Prime Ideal Principle in order to deduce that certain ideal-theoretic properties
can be ``tested'' on the set of primes. We also have a noncommutative analogue of
this result.

\begin{theorem}
\label{PIP supplement}
Let $\F$ be an Oka family such that every nonempty chain of ideals in
$\F'$ (with respect to inclusion) has an upper bound in $\F'$ (this holds,
for example, if every ideal in $\F$ is finitely generated).
\begin{itemize}
\item[\textnormal (1)] Let $\F_0$ be a semifilter of ideals in $R$. If
every prime ideal in $\F_0$ belongs to $\F$, then $\F_0 \subseteq \F$.
\item[\textnormal (2)] For an ideal $J \unlhd R$, if all prime ideals
containing $J$ (resp.\ properly containing $J$) belong to $\F$, then all
ideals containing $J$ (resp.\ properly containing $J$) belong to $\F$.
\item[\textnormal (3)] If all prime ideals of $R$ belong to $\F$, then
all ideals belong to $\F$.
\end{itemize}
\end{theorem}

\begin{proof}
Both (3) and (2) are special cases of (1). We prove (1) by a contrapositive
argument. Let $\F_0$ be a semifilter such that $\F_0 \nsubseteq \F$; we will
construct a prime ideal in $\F_0 \setminus \F$.
Let $I_0 \in \F_0 \setminus \F = \F_0 \cap \F'$. The assumption on chains in
$\F'$ allows us to apply Zorn's lemma and find $I \in \Max(\F')$ such that
$I_0 \subseteq I$. The fact that $\F_0$ is a semifilter gives $I \in \F_0$.
Because $I$ is maximal in the complement of $\F$, the Prime Ideal Principle~\ref{PIP}
implies that $I$ is prime. Thus $I$ is a prime ideal in $\F_0 \setminus \F$ as
desired.
\end{proof}

\begin{remark}
Another property of ideal families in commutative rings that was introduced in~\cite{LR} 
was that of an \emph{Ako family}; a family $\F$ in a commutative ring $R$ has the Ako
property if $(I,a), (I,b) \in \F$ implies $(I,ab) \in \F$ for any $a,b \in R$.
One can readily produce noncommutative versions of this property that satisfies the
Prime Ideal Principle. For instance, K.~Goodearl proposed the following noncommutative
version of the Ako property: 
given $I \unlhd R$ and $a, b \in R$,
\[
(I,a),(I,b) \in \F \implies (I,arb) \in \F \text{ for some } r \in R.
\]
Another possibility would be $(I,a),(I,b) \in \F \implies I+(a)(b) \in \F$.
One can readily verify that a family $\F$ satisfying either of these properties 
will also have $\Max(\F') \subseteq \Spec(R)$.
Since the Ako property has seemed to be less useful than the Oka property 
already in the commutative setting, we do not focus our current efforts on this property.
\end{remark}

Before proceeding to examples of Oka families, we conclude this section with a
discussion of a prominent family of ideals that does not generally form an Oka family.
One might wonder why a Prime Ideal Principle was given for one-sided ideals before it was
established in the more classical setting of two-sided ideals. We will give a few examples to show
that, in a sense, the case of two-sided ideals is more difficult to handle.
Specifically, we recall the following classical result of Cohen~\cite[Thm.~2]{Cohen}.

\begin{cohenthm}
An ideal in a commutative ring that is maximal with respect to not being finitely generated is
prime. Consequently, if every prime ideal of a commutative ring is finitely generated, then that
ring is noetherian.
\end{cohenthm}

This result provided strong guidance for the creation of the Prime Ideal Principle 
in~\cite{LR} and its noncommutative analogue in~\cite{Reyes}. The behavior of the family of 
finitely generated directly motivated the definition of an Oka family of ideals in a commutative 
ring and for the definition of a right Oka family in a noncommutative ring.

In~\cite[Thm.~3.8]{Reyes} we provided a version of Cohen's result above for right ideals of
a noncommutative ring. Specifically, after fixing an appropriate notion of an \emph{Oka family of
right ideals}, we proved that the family of finitely generated right ideals of any ring is an Oka family.
As a consequence we showed that a right ideal of a ring that is maximal with respect to not being
finitely generated is a completely prime right ideal. 
(Further Cohen-type results for right ideals in noncommutative rings were investigated 
in~\cite{Reyes2}.)

A fundamental difference in the case of two-sided ideals of a noncommutative ring is that,
if $R$ is not a commutative ring, then \emph{an ideal of $R$ that is maximal with respect to 
not being finitely generated need not be prime.}
This is illustrated in the following examples. 

Let $A$ and $B$ be simple rings, and let $M$ be an $(A,B)$-bimodule. The triangular matrix
ring
\[
R = T(A,M,B) = 
\begin{pmatrix}
A & M \\
0 & B
\end{pmatrix}
\]
will have exactly two prime ideals:
\begin{align*}
P_1 &= 
\begin{pmatrix}
A & M \\
0 & 0
\end{pmatrix}
= 
R
\begin{pmatrix}
1 & 0 \\
0 & 0
\end{pmatrix}
R \\
&=
\begin{pmatrix}
0 & M \\
0 & 0
\end{pmatrix}
+ R
\begin{pmatrix}
1 & 0 \\
0 & 0
\end{pmatrix}
=
\begin{pmatrix}
0 & M \\
0 & 0
\end{pmatrix}
+
\begin{pmatrix}
1 & 0 \\
0 & 0
\end{pmatrix}
R, \\
P_2 &=
\begin{pmatrix}
0 & M \\
0 & B
\end{pmatrix}
= 
R
\begin{pmatrix}
0 & 0 \\
0 & 1
\end{pmatrix}
R \\
&=
\begin{pmatrix}
0 & M \\
0 & 0
\end{pmatrix}
+ R
\begin{pmatrix}
0 & 0 \\
0 & 1
\end{pmatrix}
=
\begin{pmatrix}
0 & M \\
0 & 0
\end{pmatrix}
+
\begin{pmatrix}
0 & 0 \\
0 & 1
\end{pmatrix}
R.
\end{align*}
The expressions above make it evident that $P_1$ and $P_2$ are always finitely generated as
ideals, and that  if $M$ is finitely generated as a left $A$-module (respectively, as a right 
$B$-module), then $P_1$ and $P_2$ will be finitely generated as left (respectively, right) 
ideals of $R$.

\begin{example}
Let $K$ be a field, let $M_K$ be an infinite-dimensional vector space viewed as a bimodule 
with the same left and right action, and consider the ring $R = T(K,M,K)$ (so $A = B = K$ above). 
Then $P_1$ and $P_2$ are finitely generated as ideals, but the ideal $I = \left(
\begin{smallmatrix}
0 & V \\
0 & 0
\end{smallmatrix}
\right)$ is not finitely generated.
\end{example}

\begin{example}
Let $K/F$ be an infinite extension of fields. In the $(K,K)$-bimodule $M = K \otimes_F K$, let
$N$ be the sub-bimodule generated by $\{1 \otimes \alpha - \alpha \otimes 1 : \alpha \in K\}$.
This is a maximal sub-bimodule since $M/N \cong K$ as $(K,K)$-bimodules, and one can readily
check that $M$ is not finitely generated (by choice of $K$ and $F$).
In the ring $R = T(K,M,K)$, the choice of $M$ and $N$ guarantees that the ideal $\left(
\begin{smallmatrix}
0 & N \\
0 & 0
\end{smallmatrix}
\right)$ is maximal with respect to not being
finitely generated as an ideal of $R$. But because $R/I \cong \left(
\begin{smallmatrix}
K & K \\
0 & K
\end{smallmatrix}
\right)$, we see that $I$ is not even semiprime.
\end{example}

\begin{example}
We thank Lance Small for suggesting this example.  Let $S$ be a simple ring that is
not right noetherian (such as the endomorphism ring of a countably infinite dimensional 
vector space modulo its finite rank operators), let $Z = Z(S)$ be its central subfield, and
set $R = T(Z,S,S)$. 
Then the prime ideals $P_1$ and $P_2$ of $R$ are both finitely generated as right ideals, 
but for any right ideal $J$ of $S$ that is not finitely generated, the corresponding ideal $\left(
\begin{smallmatrix}
0 & J \\
0 & 0
\end{smallmatrix}
\right)$ of $R$ is not finitely generated.
\end{example}

\begin{remark}
\label{rem:Cohen}
Thanks to the examples given above, we see that for a general
ring $R$ the family $\F_\mathrm{fg}$ of finitely generated ideals of $R$ need not be an Oka family.
For if $\F_\mathrm{fg}$ were Oka and all prime ideals of $R$ were finitely generated, then by
Theorem~\ref{PIP supplement} all ideals of $R$ would be finitely generated, contrary
to those examples. 
Similarly, the family of ideals that are finitely generated as right
ideals is not an Oka family.
These counterexamples sharply contrast the case of commutative rings~\cite[Proposition~3.16]{LR}
and the case of Oka families of right ideals in noncommutative rings~\cite[Proposition~3.7]{Reyes}. 
\end{remark}

The examples above naturally suggest the following question: 
do there exist simple rings $A$ and $B$ and an $(A,B)$-bimodule $M$ such that $M$ is finitely 
generated as both a left $A$-module and a right $B$-module, but such that $M$ has a sub-bimodule
that is not finitely generated? For given such data, the ring $R = T(A,M,B)$ will have both prime ideals
finitely generated as left ideals and as right ideals, but will have an ideal that is not finitely generated.
Consequently, for this ring $R$ the family of ideals that are finitely generated both as left ideals and
as right ideals would fail to be an Oka family. 
We suspect that such $(A,M,B)$ should exist, but we have yet to produce an explicit example.

\section{Some families satisfying the properties $(P_i)$}
\label{P_i family section}

In this section we provide examples of families satisfying the conditions $(P_1)$,
$(P_2)$, and $(P_3)$.  Comparing this section with the next, one sees that many
well-known examples of the noncommutative ``maximal implies prime'' phenomenon 
come from such families. We begin with some of these classical examples.

Recall that a subset $S \subseteq R$ containing $1$ is called an \emph{$m$-system}
if, for every $a,b \in S$, there exists $r \in R$ such that $arb \in S$.
It is well-known (and easy to show) that an ideal $I\unlhd R$ is prime if and only if 
$R \setminus I$ is an $m$-system. The notion of $m$-system generalizes
the concept of a multiplicative subset of $R$, more common in commutative algebra.
Our first application of the Prime Ideal Principle recovers a well-known
generalization of the fact that an ideal in a commutative ring maximal with respect
to being disjoint from some fixed multiplicative set is prime.
The last part of this proposition is due to McCoy~\cite[Lem.~4]{McCoy}.

\begin{proposition}
\label{m-system}
Let $S\subseteq R$ be an $m$-system. Then the family 
$\F = \{ I \unlhd R : I \cap S \neq \varnothing \}$ is $(P_{1})$. Hence an ideal
maximal with respect to being disjoint from $S$ is prime, and every (nonzero) ideal
of $R$ intersects $S$ if and only if every (nonzero) prime ideal intersects $S$.
\end{proposition}

\begin{proof}
Because $1 \in S$, we have $R \in \F$. It is clear that $\F$ is a semifilter.
To see that $\F$ is monoidal, let $A, B \in \F$ and fix $a \in A \cap S$ and
$b \in B \cap S$. Then for some $r \in R$, $arb \in S \cap AB$. Hence
$AB \in \F$. The rest of the proposition follows from the Prime Ideal Principle~\ref{PIP}.
\end{proof}

Next we give an application of this result to the theory of algebras over a
commutative ring $k$. Let $R$ be a $k$-algebra. We will say that an element
$r \in R$ is \emph{integral over $k$} if, for all $r \in R$, there exists a
monic polynomial $f(x) \in k[x]$ such that $f(r)=0$ in $R$. The algebra $R$
is \emph{integral over $k$} if every element of $R$ is integral over $k$; if
$k$ is a field, notice that this is equivalent to $R$ being algebraic over
$k$. The next result shows that, to check whether $R$ is integral over $k$,
it suffices to check whether it is integral modulo each prime ideal.
The fact is well-known; an instance of its application in the noncommutative
setting is~\cite[Thm.~1]{Procesi}.

\begin{proposition}
Let $R$ be an algebra over a commutative ring $k$. The family
$\F = \{I \unlhd R : R/I \text{ is integral over }k \} $ is $(P_{1})$.
In particular, $R$ is integral over $k$ if and only if $R/P$ is integral over $k$ for
all (minimal) $P \in \Spec(R)$.
\end{proposition}

\begin{proof}
First, fix an element $a \in R$ and define the multiplicative set
$S_a = \{f(a) \in R : f(x) \in k[x] \text{ is monic}\} $. Then by Proposition~\ref{m-system},
the family $\F_a = \{I \unlhd R : I \cap S_a \neq \varnothing\}$ is $(P_1)$.
Notice that $\F_{a}$ is precisely the family of ideals such that the image of
$a$ in $R/I$ is integral over $k$. Now notice that the family $\F$ in question
is given by $\F = \bigcap\nolimits_{a\in R}\F_a$. Because $\F$ is an intersection
of $(P_1)$ families, $\F$ itself is $(P_1)$. The last statement now follows
from Theorem~\ref{PIP supplement}.
\end{proof}

While on the subject of polynomials, we turn our attention to PI-algebras.
We continue to assume that $k$ is a commutative ring. Let
$k\langle x_1,x_2, \ldots \rangle $ denote the ring of polynomials over $k$
in noncommuting indeterminates $x_{i}$. Suppose that $R$ is a $k$-algebra
and $f(x_1, \ldots ,x_n) \in k\langle x_i \rangle$. Then we denote
$f(R) = \{f(r_1, \ldots , r_n) \in R : r_1, \dots, r_n \in R\}$. One says that $f$
is \emph{monic} if one of its monomials of highest total degree has coefficient~1.

\begin{lemma}
If $f, g \in k \langle x_1, x_2, \dots \rangle$ are monic, then $fg$ is monic.
\end{lemma}

\begin{proof}
Suppose that $f$ has degree $m$ and $g$ has degree $n$. Let
\begin{align*}
f &= a_1 u_1 + a_2 u_2 + \cdots + a_r u_r + (\textnormal{lower degree terms}), \\
g &= b_1 v_1 + b_2 v_2 + \cdots + b_s v_s + (\textnormal{lower degree terms}),
\end{align*}
where the $\{u_i\}$ are distinct words of length $m$ and the $\{v_j\}$ are distinct words
of length $n$. We may assume that $a_1 = b_1 = 1$ because $f$ and $g$ are monic.

The highest degree terms of $fg$ are the products $a_i b_j u_i v_j$. Moreover,
the monomials $u_i v_j$ are readily seen to be distinct. Hence $fg$ has $u_1 v_1$ as
one of its leading monomials, with coefficient $a_1 b_1 = 1$, so that $fg$ is monic.
\end{proof}

A $k$-algebra $R$ is a \emph{polynomial identity algebra} (or \emph{PI algebra}) over 
$k$ if $f(R) = 0$ for some
monic polynomial $f \in k\langle x_1,x_2,\ldots\rangle$ (see, for instance, \cite[6.1.2]{Rowen}).
The last sentence of the following is a folk result; for example, it is used implicitly in
the proof of~\cite[Thm.~2]{AmitsurSmall}.  

\begin{proposition}
Let $k$ be a commutative ring and let $S \subseteq k\langle x_1,x_2, \ldots\rangle$ be an 
$m$-system of polynomials in noncommuting indeterminates. For a $k$-algebra $R$, the family
$\F = \{I \unlhd R : R/I \text{ satisfies some } f \in S\}$ is $(P_1)$.
Hence, an ideal maximal with respect to $R/I$ not being a PI-algebra over $k$ is prime.
\end{proposition}

\begin{proof}
Notice that $\F = \{I \unlhd R : f(R) \subseteq I \text{ for some } f \in S\}$.
This makes it clear that $\F$ is a semifilter. To see that $\F$ is monoidal,
suppose that $A,B \in \F$. Choose $f,g \in S$ and
$h \in k\langle x_1, x_2, \ldots\rangle$ such that $f(R) \subseteq A$
and $g(R) \subseteq B$ as well as $fhg \in S$. Then
$(fhg)(R) \subseteq f(R)h(R)g(R) \subseteq AB$ while $fhg\in S$, proving that
$AB \in \F$.

For the last statement, we apply Theorem \ref{PIP} to the family $\F$
taking $S$ to be the set of all monic polynomials in $k\langle x_1,x_2, \ldots\rangle$,
which is multiplicatively closed by the lemma above.
\end{proof}

We recall another familiar result from commutative algebra. For a module
$M_R \neq 0$ over a commutative ring $R$, a \emph{point annihilator} is an
ideal $I\unlhd R$ of the form $I=\ann(m)$ for some $m \in M \setminus\{0\}$.
Then a maximal point annihilator of $M$ is prime. But notice that, because $R$
is commutative, we have $\ann(m) = \ann(mR)$ for all $m\in M$. So a maximal
point annihilator of $M$ is just an ideal maximal
among the ideals of the form $\ann(N)$ where $0 \neq N_R \subseteq M$. From
this perspective, the result generalizes to modules over noncommutative rings: 
for a module $M_R \neq 0$ over a ring $R$, an ideal maximal among the annihilators 
$\ann(N)$ where $0 \neq N_R \subseteq M$ is prime.

This well-known fact is widely attributed to I.\,N.~Herstein; for instance,
see~\cite[Thm.~6]{Kaplansky74}
(Lance Small has communicated to us that the first instance of this result is
likely in~\cite{Herstein}.)
We recover it via the Prime Ideal Principle below.

\begin{proposition}
\label{max annihilator}
Fix a module $M_{R}\neq 0$. Then the family $\F = \{J \unlhd R : NJ=0
\text{ for } N_R \subseteq M \implies N = 0\}$ is $(P_1)$. A
maximal annihilator of a nonzero submodule of $M$ is prime.
\end{proposition}

\begin{proof}
It is straightforward to see that $\F$ is a semifilter. So suppose
that $A,B \in \F$ and that $N_R \subseteq M$ such that $N(AB) = 0$.
Then $(NA)B = 0 \implies NA = 0 \implies N = 0$.

To see that a maximal annihilator is prime, we only need to note that the
maximal annihilators of nonzero submodules of $M$ are precisely the elements
of $\Max(\F')$ and then apply the Prime Ideal Principle~\ref{PIP}.
\end{proof}

Applying this to the module $R_R$ gives the following result. 

\begin{proposition}
In a nonzero ring $R$, the family $\F$ of ideals $I$ that are faithful
as left $R$-modules is a $(P_1)$ family. If $R$ satisfies the ascending chain
condition (ACC) on ideals, then $R$ is a prime ring if and only if every nonzero 
prime ideal of $R$ is left faithful.
\end{proposition}

\begin{proof}
Applying Proposition~\ref{max annihilator} to $M_R = R_R$, we see that the
family $\F$ there corresponds to the family of left-faithful ideals
of $R$. Turning our attention to the final statement, notice that the
``only if" part holds without the ACC assumption. The ``if" part then follows
from Theorem~\ref{PIP supplement}.
\end{proof}

As a variation on the theme of maximal annihilator ideals, there are also
maximal ``middle annihilators." For ideals $X,Y \unlhd R$ we use the notation
$\ann(X,Y) = \{r \in R : XrY=0\}$, which is certainly an ideal of $R$.
We call such an ideal a \emph{middle annihilator} of $R$ if $XY \neq 0$.
This definition originated in unpublished work of Kaplansky
(see~\cite[p. 782]{Krause}), who proved the following ``maximal implies
prime'' result. 

\begin{proposition}
For any ring $R$, the family $\F = \{I \unlhd R : \mbox{for all } X,Y \unlhd R, \  XIY=0 
\implies XY=0\}$ is $(P_1)$. A maximal middle annihilator of $R$ is
prime.
\end{proposition}

\begin{proof}
Certainly $\F$ is a semifilter. Suppose that $A,B \in \F$, and let
$X,Y \unlhd R$ such that $X(AB)Y=0$. Then $(XA)BY=0$ with $B \in \F$
implies that $XAY=0$. But then $A \in \F$, so we must have $XY=0$. Thus
$AB \in \F$, proving that $\F$ is $(P_1)$. Because $\Max(\F')$ is
equal to the set of maximal middle annihilators of $R$, the last statement
follows from the Prime Ideal Principle~\ref{PIP}.
\end{proof}

Our next aim is to recover the well-known result that if the ideals in a ring
$R$ satisfy ACC, then $R$ has finitely many minimal primes, along with a
noncommutative generalization of a result of D.\, D.~Anderson from~\cite{Anderson}
that any commutative ring in which the minimal primes are finitely generated has
finitely many minimal primes.
This makes use of the following.

\begin{lemma}
\label{right generating sets}
Let $I$ be a right ideal and $J$ be an ideal of a ring $R$.  Suppose that these
are generated as right ideals by $X \subseteq I$ and $Y \subseteq J$.
Then the product $IJ$ is generated as a right ideal by the set $\{xy : x \in X,\ y \in Y\}$.
In particular, if $I$ and $J$ are finitely generated (resp.\ principal) as right ideals,
then $IJ$ is finitely generated (resp.\ principal) as a right ideal.
\end{lemma}

\begin{proof}
The second claim clearly follows from the first. 
Given $I$ and $J$ as above, the fact that $J$ is a (left) ideal gives the second
equality below:
\begin{align*}
IJ &= \left( \sum\nolimits_{x \in X} xR \right) J \\
&= \sum\nolimits_{x \in X} xJ \\
&= \sum\nolimits_{x \in X} x \left( \sum\nolimits_{y \in Y} yR \right) \\
&= \sum\nolimits_{(x,y) \in X \times Y} xy R. \qedhere
\end{align*}
\end{proof}

The following statement under hypothesis~(1) is well known; for instance,
see~\cite[Theorem~3.4]{GoodearlWarfield}.
The version of Anderson's theorem under hypothesis~(2) was first suggested to us by T.\,Y.~Lam.

\begin{proposition}
For any monoidal set $\setS$ of ideals in $R$, the family
$\F = \{I \unlhd R : I\supseteq J \text{ for some } J \in \setS\}$ is 
$(P_1)$. Suppose that one of the following two conditions holds:
\begin{enumerate}
\item $R$ satisfies ACC on ideals;
\item The minimal prime ideals of $R$ are finitely generated as right ideals.
\end{enumerate}
Then the zero ideal is a product of (minimal) prime ideals, and $R$ has finitely
many minimal primes.
\end{proposition}

\begin{proof}
It is clear that the family $\F$ is $(P_1)$.
To prove the last statement, let $\setS$ be the set of all finite products
of minimal prime ideals of $R$ and let $\F$ be the $(P_1)$ family described
in the Proposition. Notice that $\Spec(R) \subseteq \F$ because every prime
contains a minimal prime.
In the next paragraph, we will show that every nonempty chain in $\F'$
has an upper bound in $\F'$. Theorem~\ref{PIP supplement} thus implies
that $\F$ consists of all ideals. In particular $0 \in \F$, so that the zero
ideal is a product of minimal primes, say $P_1 \cdots P_n = 0$. For any minimal
prime $P \unlhd R$, we have $P \supseteq 0 = P_1 \cdots P_{n}$. So some
$P_i \subseteq P$, giving $P=P_i$. Thus the minimal primes of $R$ are precisely
$P_1, \dots, P_n$.

It remains to verify that every nonempty chain in $\F'$ has an upper bound.
This is trivial in case~(1) holds. If $R$ satisfies~(2) then 
Lemma~\ref{right generating sets} implies that every ideal in $\setS$ is
finitely generated as a (right) ideal, whence the claim easily follows from Zorn's
lemma.
\end{proof}

Another notion that gives rise to prime ideals is the Artin-Rees property.
An ideal $I \unlhd R$ is said to have the \emph{right Artin-Rees property}
if, for any right ideal $K_R \subseteq R$, there exists a positive integer $n$ 
such that $K \cap I^n \subseteq KI$.
This property arises in the study of localization in noetherian rings
(for example, see~\cite[Ch.~13]{GoodearlWarfield}
or~\cite[\S 4.2]{McConnellRobson}). The ring $R$ is said to be a \emph{right
Artin-Rees ring} if every ideal of $R$ has the right Artin-Rees property.

(In commutative algebra, the Artin-Rees Lemma is a result about $I$-filtrations
of finitely generated $R$-modules, where $R$ is commutative noetherian and
$I \unlhd R$---see~\cite[Lem.~5.1]{Eisenbud}. In particular, the lemma implies
that every ideal of  a commutative noetherian ring satisfies the Artin-Rees property.
Thus the family of ideals satisfying the right Artin-Rees property is mostly of interest
in the noncommutative case.)

Parts of the following result were previously observed in~\cite[\S 3]{Sirola}.
 However, the ``maximal implies prime" result seems to be new.

\begin{proposition}
The family of ideals with the right Artin-Rees property in a ring $R$ satisfies $(P_2)$.
Thus an ideal of $R$ maximal with respect to not having the right Artin-Rees 
property is prime.
In particular, if $R$ satisfies ACC on ideals, then $R$ is a right Artin-Rees ring
if and only if every prime ideal of $R$ has the right Artin-Rees property.
\end{proposition}

\begin{proof}
Let $\F$ be the family of ideals in $R$ with the right Artin-Rees property.
Suppose that $J \in \F$ and $I \unlhd R$ with $I^2 \subseteq J \subseteq I$.
Fix a right ideal $K_R \subseteq R$ and an integer $n$ such that
$K \cap J^n \subseteq KJ$. Then $K \cap I^{2n} \subseteq K \cap J^n
\subseteq KJ \subseteq KI$, which proves that $I \in \F$.

Next suppose that $A,B \in \F$. Fix a right ideal $K_R \subseteq R$,
and choose integers $m$ and $n$ such that $K \cap A^m \subseteq KA$ and
$(KA) \cap B^n \subseteq (KA)B$. Then for $N = \max\{m,n\}$, we have 
\begin{align*}
K \cap (AB)^N &= K \cap (AB)^m \cap (AB)^n \\
 & \subseteq (K \cap A^m) \cap B^n \\
 & \subseteq (KA) \cap B^n \\
 & \subseteq KAB.
\end{align*}
Thus $AB \in \F$ and $\F$ is monoidal. So $\F$ satisfies $(P_2)$. The
last two statements follow from Theorems~\ref{PIP} and~\ref{PIP supplement}.
\end{proof}

The next family that we consider is the family of ideal direct summands of
$R$; that is, the family of all ideals $I \unlhd R$ such that $R=I \oplus J$
for some ideal $J \unlhd R$. It is well-known that an ideal $I\unlhd R$ is such
a direct summand if and only if $I = (e)$ for some central idempotent $e$ of $R$, in
which case $J = (1-e)$ (see, for instance,~\cite[Ex.~1.7]{ExercisesClassical}).
In fact, we will work in the slightly more general setting of multiplicative
sets of central idempotents.

\begin{proposition}
Let $R$ be a ring and let $S \subseteq R$ be a multiplicative set of central
idempotents. Then the family $\F = \{(e) \unlhd R : e \in S\}$ satisfies $(P_2)$.
In particular, the family of ideal direct summands of $R$ is $(P_3)$. An ideal
maximal with respect to not being an ideal direct summand is maximal. A ring
$R$ is a finite direct product of simple rings if and only if every ideal of $R$ is
an ideal direct summand, if and only if all maximal ideals of $R$ are ideal direct
summands.
\end{proposition}

\begin{proof}
To prove that $\F$ is monoidal, let $e,e' \in S$ so that $(e)$ and
$(e')$ are arbitrary ideals in $\F$.  Then $ee' \in S$, so we have
$(e)(e') = (ee') \in \F$. Also, if $I \unlhd R$ with $(e)^2 \subseteq I \subseteq (e)$,
then $(e)^2 = (e^2) = (e)$ implies that $I = (e) \in \F$. Thus $\F$ satisfies
$(P_2)$.

Taking $S$ to be the multiplicative set of \emph{all} central idempotents
of $R$, the family $\F$ becomes the set of ideal direct summands of $R$.
Now let $M \in \Max(\F')$. By the Prime Ideal Principle~\ref{PIP}, $M$ is prime.
In the factor ring $R/M$, every ideal is generated by a central idempotent.
We want to prove that $R/M$ is simple.
Let $0 \neq I \unlhd R/M$ with $I = (e)$ for some central idempotent $e \in R/M$.
Then $R/M = (e) \oplus (1-e)$, and $(e)(1-e) = (0)$. Because $R/M$ is prime
and $(e) \neq 0$, we must have $(1-e)=0$. Thus $R/M = (e) = I$, proving that
$R/M$ is simple. 

The equivalence of the condition that all ideals are direct summands with the
condition that all maximal ideals are summands follows from Theorem~\ref{PIP supplement}
along with the fact that every ideal in $\F$ is finitely generated.
It remains to show that, if all ideals are direct summands, then $R$ is a finite direct
product of simple rings (the converse being clear). If all ideals of $R$ are direct summands, 
then $R$ is a finitely generated semisimple module over
the enveloping ring $R^e = R \otimes_{\Z} R^{\mathrm{op}}$. By the classical theory of
semisimple modules (see~\cite[\S 2]{FC}), this means that $R$ is a direct sum
of simple $R^e$-submodules. But a simple $R^e$-submodule is a minimal ideal,
and the desired equivalence easily follows.
\end{proof}

Next we investigate the ideals $I \unlhd R$ such that $R/I$ is Dedekind finite.
(Recall that a ring $R$ is \emph{Dedekind finite}, alternatively called \emph{directly finite}
or \emph{von Neumann finite}, if $ab=1$ implies $ba=1$ for any $a,b \in R$.)

\begin{proposition}
The family $\F = \{I \unlhd R : R/I \text{ is Dedekind finite}\}$ is
$(P_3)$. An ideal maximal with respect to $R/I$ not being Dedekind finite is
prime.
\end{proposition}

\begin{proof}
Let $A,B \in \F$ and $I \unlhd R$ with $AB \subseteq I\subseteq A\cap B$. We
want to show that $I \in \F$. This is equivalent to showing that if $a,b \in R$
with $1-ab \in I$, then $1-ba \in I$. If $1-ab \in I \subseteq A \cap B$ then
we must have $1-ba \in A \cap B$ because $A,B \in \F$. But then $1-2ba+baba =
(1-ba)^2 \in AB \subseteq I$. Also notice that $ba-baba = b(1-ab)a \in I$
because $1-ab\in I$. Thus
\[
1-ba = (1-2ba+baba) + (ba-baba) \in I,
\]
proving that $I\in \F$. The last statement now follows from Theorem~\ref{PIP}.
\end{proof}

Consideration of flat modules leads to the following example. Notice that
the conclusion here that the family is ($P_3$) is stronger than the one given 
in~\cite[(5.7)]{LR} when restricted to the commutative case.

\begin{proposition}
Let $R$ be a ring and set $\F = \{I \unlhd R : (R/I)_R \text{ is flat}\}$.
Then $\F$ is a $(P_3)$ family. In particular, $\F$ is monoidal and an ideal
$P$ maximal with respect to $(R/P)_R$ not being flat is a prime ideal.
\end{proposition}

\begin{proof}
Applying~\cite[(4.14)]{Lectures} to the exact sequence $0 \to I \to R \to
R/I \to 0$, we see that $(R/I)_R$ is flat if and only if, for every left ideal
$_{R}L \subseteq R$ one has $I \cap L \subseteq IL$ (so that $I \cap L = IL$).
Thus 
\[
\F = \{I \unlhd R : I \cap L \subseteq IL \text{ for all } {}_R L \subseteq R\}.
\]
Suppose that $I, A, B \unlhd R$ with $A, B \in \F$ and $AB \subseteq I \subseteq A \cap B$.
Then $I\in \F$ because, for every ${}_R L \subseteq R$, we have
\begin{align*}
I \cap L &\subseteq (A \cap B) \cap L \\
&= A \cap (B \cap L) \\
&= A \cap (BL) \\
&= (AB)L \\
&\subseteq IL.
\end{align*}
The last two statements follow from the the fact that $(P_3)$
families are monoidal and from the Prime Ideal Principle~\ref{PIP}.
\end{proof}

\section{Oka families and families constructed from module categories}
\label{Oka family section}

In this final section, we provide examples of families that satisfy the $r$-Oka and
Oka properties. Beginning with Proposition~\ref{right module categories}, these 
examples are produced from classes of modules that satisfy certain properties. 
Some of those examples in fact satisfy the strongest property ($P_1$), but they are
proved in this section as applications of the module-theoretic methods to be
introduced.

Kaplansky showed~\cite[p.~486]{Kaplansky49} that a commutative ring $R$ is a principal
ideal ring (that is, every ideal of $R$ is principal) if and only if every prime ideal
of $R$ is principal. Here we give a certain generalization of this
result to noncommutative rings. (A different generalization for one-sided
ideals was given in~\cite[Theorem~5.11]{Reyes2}.) We will say that an element $x \in R$ is 
\emph{normal} if $xR=Rx$, in which case we have $xR = (x) = Rx$.
This is equivalent to saying that both inclusions $xR \subseteq (x)$ and
$Rx \subseteq (x)$ are equalities. The set of normal elements is multiplicatively
closed: if $x,y \in R$ are normal, then $xyR=xRy=Rxy$.

Let us refine some of our notation. For a right ideal $I_{R} \subseteq R$ and
an element $x\in R$, we define $x^{-1}I = \{y \in R : xy \in I \}$ which
is \emph{a priori} only a right ideal of $R$. However, if $I$ is a two-sided
ideal and $x$ is normal, then one can easily verify that $x^{-1}I = (x)^{-1}I$, in
which case $x^{-1}I $ is an ideal.

\begin{proposition}
\label{prop:principal normal}
If $S$ is a multiplicative set of normal elements in a ring $R$,
then the set $\F = \{(s) : s \in S \}$ is a strongly $r$-Oka family.
Every ideal of a ring $R$ is generated by a single normal element if and only if
every prime ideal in $R$ is generated by a normal element.
\end{proposition}

\begin{proof}
Let $I,J \unlhd R$ such that $(I,J),J^{-1}I \in \F$. Choose $x,y \in S$
such that $(I,J) = (x)$ and $J^{-1}I=(y)$. We claim that $I = (xy)$;
because $xy \in S$ we will therefore have $I \in \F$.
Noticing that $I \subseteq (I,J) = (x) = xR$, it is
easy to verify that $I = x(x^{-1}I)$. But $x$ is normal, giving the first
of the following equalities: $x^{-1}I = (x)^{-1}I = J^{-1}I = (y)$.
Once more, normality of $x$ gives the final equality in $I = x(x^{-1}I) = x(y) = (xy)$.

For the last statement, we apply Theorem~\ref{PIP supplement} to the family
$\F$ taking $S$ to be the set of \emph{all} normalizing elements.
\end{proof}

In the special case of a right noetherian ring, for an ideal to be principally generated
by a normal element, it suffices to test separately that it is principal as a left and a right
ideal.
The following may be a known folk result, but we became aware of it in a conversation with 
Konstantin Ardakov on Mathematics Stack 
Exchange.\footnote{See \url{http://math.stackexchange.com/q/627609}}

\begin{lemma}
\label{lem:left right principal}
Let $I$ be an ideal of a ring $R$ such that $I = aR = Rb$ for some $a,b \in R$. Assume that
one of the following hypotheses holds:
\begin{enumerate}
\item $R$ is right noetherian;
\item $R$ satisfies ACC on principal right ideals and $aR$ is a projective right ideal
(e.g., the right annihilator of $a$ is zero).
\end{enumerate}
Then $I = bR$ and $b$ is normal.
\end{lemma}

\begin{proof}
Fix $x,y \in R$ such that $a = xb$ and $b = ay$.  
Because $I$ is a left ideal, the function $\lambda(r) = xr$ restricts
to a right $R$-module endomorphism of $I = aR$. This map is surjective because
$a = xb = xay$. In both cases~(1) and~(2), we will prove that $\lambda$ is an 
automorphism.  Because $\lambda(b) = a$ implies that the submodule $bR \subseteq aR$ 
also has image equal to $\lambda(bR) = aR$, injectivity of $\lambda$ will in turn
imply that $bR = aR$ as desired. 

Under hypothesis~(1), the right ideal $I$ is a noetherian (and therefore Hopfian) module, 
whence this surjection is an automorphism. Now suppose~(2) holds, and assume toward a
contradiction that $\lambda$ has kernel $0 \neq K_1 \subseteq aR$. We can inductively
define a chain of right ideals $K_1 \subsetneq K_2 \subsetneq \cdots$ by setting $K_n/K_{n-1}$
to be the (nonzero) kernel of the homomorphism $aR/K_{n-1} \cong aR \overset{\lambda}{\to} aR$.
The third isomorphism theorem thus yields short exact sequences
\[
0 \to K_n \to aR \to aR \to 0
\]
for all $n$. Because $aR$ is projective, each of these sequences splits. Thus $K_n$, being a 
direct summand of the principal right ideal $aR$, is itself principal. Therefore the strictly
ascending chain of principal right ideals $K_n$ contradicts condition~(2).
\end{proof}

\begin{corollary}
Let $R$ be a ring such that either:
\begin{enumerate}
\item $R$ is right noetherian, or
\item all principal right ideals of $R$ are projective (e.g., $R$ is a domain) and $R$
satisfies ACC on principal right ideals.
\end{enumerate}
Then any multiplicative submonoid of the family
\[
\F = \{ I \unlhd R \mid I = aR = Rb \mbox{ for some } a,b \in R\}
\]
(including $\F$ itself) is an $r$-Oka family, and every ideal of $R$ is generated by a single 
normal element if and only if every prime ideal of $R$ is principal as both a left and right
ideal.
\end{corollary}

\begin{proof}
By Lemma~\ref{right generating sets}, the family $\F$ is closed under products.
Next, Lemma~\ref{lem:left right principal} implies that every element
of $\F$ is generated by a single normal element. The result now follows from
Proposition~\ref{prop:principal normal}.
\end{proof}

We also mention the following immediate corollary of Lemma~\ref{lem:left right principal}.

\begin{corollary}
If $R$ is a principal left and right ideal ring, then every ideal of $R$ is generated by a
single normal element.
\end{corollary}

Now we turn our attention to invertible ideals. Let $R$ be a subring of a ring $Q$.
We will say that an ideal $I \unlhd R$ is \emph{invertible with respect to $Q$} 
(or $Q$-invertible) if 
there exists an $(R,R)$-subbimodule $I^* \subseteq Q$ such that $I \cdot I^* = I^* \cdot I =R$.
Certainly $I^*$ is unique.
We claim that every invertible ideal in $R$ is finitely generated as a right
ideal. To see this, let $I \unlhd R$ be invertible and choose $x_i \in I$ and
$y_i \in I^*$ such that $\sum_{i=1}^{n} x_i y_i = 1$. Then for $x \in I$
we have $x = \sum x_i (y_i x)$ where each $y_i x \in I^* I = R$. Thus
$I = \sum x_i R$ is finitely generated as a right ideal. A symmetric argument 
proves that an invertible ideal is also finitely generated as a left ideal.

In the argument below, the reader is advised to take care not to confuse products
$J^* I$ of $(R,R)$-bimodules inside of $Q$ with ideal quotients
$J^{-1} I = \{ x \in R : xJ \subseteq I \}$.

\begin{lemma}
\label{factorization lemma}
Let $R \subseteq Q$ be as above, and let $I,J \unlhd R$ where $I \subseteq J$
and $J$ is invertible. Then $I = J \cdot (J^{-1}I)$.
\end{lemma}

\begin{proof}
Notice that $I \subseteq J$ implies that $J^* I \subseteq J^* J = R$ is an
ideal of $R$. Because $J(J^*I) = I$, it follows that $J^*I \subseteq J^{-1}I$. Left
multiplying by $J$ gives $I = J(J^* I) \subseteq J (J^{-1}I) \subseteq I$, and the
desired equality follows.
\end{proof}

\begin{proposition}
Let $R$ be a subring of a ring $Q$. Then any monoidal family $\F$ of $Q$-invertible
ideals of $R$ is a strongly $r$-Oka family. Hence an ideal maximal with respect
to not being invertible is prime.
Every nonzero ideal of $R$ is invertible if and only if every nonzero prime ideal
of $R$ is invertible.

If every nonzero ideal of $R$ is $Q$-invertible, then the following hold:
\begin{enumerate}
\item $R$ is a prime ring;
\item If every nonzero ideal of $Q$ has nonzero intersection with $R$, then
$Q$ is simple.
\end{enumerate}
\end{proposition}

\begin{proof}
Let $I,J \unlhd R$ such that $(I,J),J^{-1}I \in \F$. Noting that
$(I,J)^{-1}I = J^{-1}I$, we may apply Lemma~\ref{factorization lemma} to
find that $I = (I,J) \cdot (J^{-1}I) \in \F$.
The second statement follows from Theorem~\ref{PIP supplement} once we notice that
the family of \emph{all} $Q$-invertible ideals in $R$ is monoidal and every
invertible ideal is finitely generated (as a right ideal). 

Now assume that all nonzero ideals of $R$ are $Q$-invertible.
Then~(1) follows from the Prime Ideal Principle~\ref{PIP} because the zero ideal is 
maximal in the complement of the Oka family of all invertible ideals. 
Now assume that the hypothesis of~(2) holds. To see
that $Q$ is simple, let $K \unlhd Q$ be nonzero. Then $I = K \cap R$ is a nonzero 
ideal in $R$. Thus $I$ is invertible and consequently $1 \in R = II^* \subseteq KI^* \subseteq K$. 
So $K=Q$, proving that $Q$ is simple.
\end{proof}

Invertibility is typically studied in the context where $Q$ is a ring of quotients.
For instance, let $R$ be a right Ore ring (that is, a ring in which the set of 
non-zero-divisors is a right Ore set~\cite[\S10B]{Lectures}) and let $Q$
be its classical right ring of quotients. (In this case, every right ideal of $Q$ has
nonzero intersection with $R$, so the hypothesis of~(B) above is satisfied.)
If every nonzero prime ideal of $R$ is $Q$-invertible, the above implies that 
all nonzero ideals of $R$ are $Q$-invertible, that $R$ is prime, and that $Q$ is simple. 
When these conditions hold, $Q$ is (simple) artinian if and only if $R$ has finite right 
uniform dimension. Indeed, the simple ring $Q$ is equal to its own classical right ring 
of quotients. As it is prime and nonsingular, Goldie's theorem states that it is (right)
artinian exactly when $Q$ has finite right uniform dimension. But $Q$ has finite 
right uniform dimension if and only if $R$ does (see~\cite[(10.35)]{Lectures}).
(Of course, the ring of quotients $Q$ need not always be artinian if all ideals of $R$ are
$Q$-invertible. For instance, consider that a ring which is a union of simple artinian 
subrings is again simple and equal to its own classical ring of quotients, but may fail 
to be artinian. Specifically we may take $R = Q$ to be the direct limit of the matrix rings 
$R_n = \mathbb{M}_{2^{n}}(k)$ for some division ring $k$, with $R_n$ considered 
as a subring of $R_{n+1}$ via the natural map 
$R_n \rightarrow R_n \otimes I_2 \subseteq R_n \otimes_k \mathbb{M}_2(k) \cong R_{n+1}$.
An explicit argument that $Q$ is not artinian is given in~\cite[p.~40]{FC}.)

The remainder of the results in this section are based on families of ideals in a ring
$R$ that are formed by considering those ideals $I$ such that the $(R,R)$-bimodule
$R/I$ lies in a class $\C$ of bimodules, \emph{always assumed to contain the zero
module}, satisfying certain properties. We say that a
class $\C$ of $(R,R)$-bimodules is \emph{closed under extensions} if, for every
short exact sequence of bimodules $0 \to L \to M \to N \to 0$, if $L,N \in \C$ then
$N \in C$. (Taking $L = 0$, note that this implies that if $M \cong N$ and $N \in \C$
then also $M \in \C$.) We say that $\C$ is \emph{closed under quotients} if every
homomorphic image of a bimodule in $\C$ is also in $\C$.

The first application of this kind of reasoning gives a method to construct examples of
($P_1$) and $r$-Oka families. As one can tell from the hypotheses of~(2) and~(3) below,
it applies especially well to right noetherian rings. 
In the following, for a module $M$ we use the term \emph{direct sum power of $M$}
to mean a module that is a direct sum of copies of $M$, that is, a module
of the form $M^{(X)} = \bigoplus_X M$ for some set $X$.

\begin{proposition}
\label{right module categories}
Let $R$ be a ring and $\C$ a class of right $R$-modules containing the zero module, and 
denote $\F = \{I \unlhd R : R/I \in \C\}$.
\begin{enumerate}
\item If $\C$ is closed under extensions, quotients, and arbitrary
direct sum powers, then $\F$ satisfies $(P_1)$.
\item Suppose that every ideal in $\F$ is finitely generated as a right ideal and that $\C$
is closed under extensions and quotients. Then $\F$ satisfies $(P_1)$. 
\item Suppose that every principal (equivalently, each finitely generated) ideal in $R$ is 
finitely generated as a right ideal and that $\C$ is closed under extensions and quotients. 
Then $\F$ is an $r$-Oka semifilter.
\end{enumerate}
\end{proposition}

\begin{proof}
In each case, $\F$ is a semifilter because $\C$ is closed under quotients, and $R \in \F$
because $0 \in \C$.
First suppose that $\C$ is as in~(1). Let $I,J \in F$ and consider the short exact 
sequence  of right $R$-modules 
\[
0 \to \frac{I}{IJ} \to \frac{R}{IJ} \to \frac{R}{I} \to 0.
\]
The right module $I/IJ$ is annihilated by $J$, so for some sufficiently large set $X$ we have a surjection of
$\bigoplus\nolimits_{X} R/J$ (which is in $\C$ by closure under direct
sum powers) onto $I/IJ$ (which we now see to be in $\C$ by closure under
quotients). Because $I/IJ$ and $R/I$ belong to $\C$, which is closed
under extensions, we have $R/IJ \in \C$ and thus $IJ \in \F$.

Next suppose that $\F$ and $\C$ are as in~(2), and fix $I,J \in \F$. The argument 
that $IJ \in \F$ is identical to that for~(1), taking into account the following modificiations.
Because of the hypothesis on $\F$, $I \in \F$ implies that $I/IJ$ is finitely generated 
as a right $R$-module. Thus the set $X$ above can be made finite, and because 
$\C$ is closed under extensions we obtain that  $(R/J)^X$ lies in $\C$.

Now under the hypothesis of~(3), let $I \unlhd R$ and $a \in R$ be such that
$(I,a),(a)^{-1}I \in \F$. The assumption on $R$ implies that $(I,a)/I$ is a 
finitely generated right module in the short exact sequence 
\[
0 \to \frac{(I,a)}{I} \to \frac{R}{I} \to \frac{R}{(a)^{-1}I} \to 0.
\]
Further, the right annihilator of $(I,a)/I$ is equal to $(a)^{-1}I$. So we have
a surjection of $R$-modules $(R/(a)^{-1}I)^n \twoheadrightarrow (I,a)/I$ for 
some positive integer~$n$, and the argument can proceed as above to conclude
that $I \in \F$. Hence $\F$ is an $r$-Oka family.
\end{proof}

\begin{example}
Given an $m$-system $S$ in a ring $R$, let $\C$ be the class of right $R$-modules
$M$ such that $\ann(M) \cap S \neq \varnothing$.
Then $\C$ is closed under extensions and quotients, because in any exact sequence 
$0 \to L \to M \to N \to 0$ of right $R$-modules, one has 
$\ann(N) \ann(L) \subseteq \ann(M) \subseteq \ann(N)$.  Also,
$\ann(\bigoplus_{X}M) = \ann(M)$ for any indexing set $X$. So $\C$ satisfies
the hypothesis of Proposition~\ref{right module categories}(1) and satisfies $(P_1)$. 
Because $\ann((R/I) _R) = I$ for any $I \unlhd R$, the family $\F = \{I \unlhd R : R/I \in \C\}$
is equal to the set of all ideals that have nonempty intersection with $S$. Thus we recover
the same family that we considered in Proposition~\ref{m-system}. 
\end{example}

\begin{example}
Let $R$ be a ring, and let $S \subseteq R$ be a multiplicative subset. Let $\C$
be the class of \emph{$S$-torsion modules}: those modules in which every element
is annihilated by an element of $S$. This class is easily seen to be closed under
extensions, quotients, and arbitrary direct sums. It follows from 
Proposition~\ref{right module categories}(1) that the corresponding family 
\[
\F = \{I \unlhd R : R/I \in \C\} = \{I \unlhd R : \mbox{for all } r \in R,\ r^{-1}I \cap S \neq \varnothing\}
\]
is $(P_1)$.
\end{example}

Two familiar facts in noncommutative ring theory state that (i)~a ring is right
artinian if and only if it is right noetherian and all of its prime factors are right
artinian, and (ii)~a left artinian ring is right artinian if and only if
it is right noetherian. Here we improve these criteria for a ring to be right
artinian by relaxing the right noetherian hypothesis in each case to the
mere assumption that the \emph{two-sided} ideals of the ring be finitely
generated as right ideals.

\begin{proposition}
\label{when R is right artinian}
If every principal ideal of a ring $R$ is finitely generated as a right ideal, 
then the family $\F = \{I \unlhd R : R/I \text{ is right artinian}\}$ is $r$-Oka.
A ring $R$ is right artinian if and only if $R/P$ is right artinian for all $P \in \Spec(R)$
and every ideal in $R$ is finitely generated as a right ideal. A left artinian ring is right 
artinian if and only if all of its ideals are finitely generated as right ideals.
\end{proposition}

\begin{proof}
The fact that $\F$ is $r$-Oka follows from Proposition~\ref{right module categories}(3)
applied to the class $\C$ of artinian right $R$-modules.

If $R$ is right artinian, then certainly all of its prime factors are right
artinian. Also $R$ is right noetherian, so its ideals are all finitely generated as 
right ideals. For the converse, suppose that all of the prime factors of $R$ are right
artinian and that its ideals are finitely generated as right ideals. By the first part of this
proposition, the family $\F$ defined above is strongly $r$-Oka. Because
$\Spec(R) \subseteq \F$, the finiteness assumption on the ideals of $R$
allows us to conclude that all ideals lie in $\F$, according to Theorem~\ref{PIP supplement}.
In particular we have $0 \in \F$, and $R$ is right artinian.

For the last statement, given a left artinian ring $R$ whose ideals are
finitely generated as right ideals, we merely need to show that $R/P$ is right artinian for all
$P \in \Spec(R)$. For each prime ideal $P \unlhd R$ the factor ring $R/P$ is
left artinian and prime, hence semisimple and right artinian.
\end{proof}

An even more general criterion for a left artinian ring to be right artinian is 
the result of D.\,V.~Huynh~\cite[Theorem 1]{Huynh} that a left artinian ring
is right artinian if and only if $J/(J^{2}+D)$ is right finitely generated where $J \unlhd R$
is the Jacobson radical of $R$ and $D \unlhd R$ is the maximal divisible torsion
ideal. 

A similar consideration gives a criterion for a ring to be right noetherian.
We do not include the proof, which is identical to the proposition above.
Notice that we do not have a corresponding result telling when a left
noetherian ring is right noetherian due to the existence of prime left
noetherian rings which are not right noetherian, such as a twisted polynomial
extension $k[x; \sigma]$  of a field $k$ by an endomorphism $\sigma \colon k \to k$
that is not surjective.

\begin{proposition}
\label{when R is right noetherian}
If every principal ideal of a ring $R$ is finitely generated as a right ideal, then the
family $\F = \{I \unlhd R : R/I \text{ is right noetherian}\}$ is
$r$-Oka. A ring $R$ is right noetherian if and only if $R/P$ is right noetherian for
all $P \in \Spec(R)$ and every ideal in $R$ is finitely generated as a right ideal.
\end{proposition}

The previous two results raise the question of whether the assumption on the
ideals of $R$ can be further weakened. The following example shows that we
cannot hope to do much better than requiring every ideal to be right
finitely genrated. Consider the ring $R = \left(
\begin{smallmatrix}
\mathbb{R} & \mathbb{R} \\ 
0 & \mathbb{Q}
\end{smallmatrix}
\right)$. One can show (using, for example,~\cite[(1.17)]{FC}) that the
only ideals in $R$ are 
\[
0,\ I =
\begin{pmatrix}
0 & \mathbb{R} \\ 
0 & 0
\end{pmatrix}
,\ P_{1} =
\begin{pmatrix}
\mathbb{R} & \mathbb{R} \\ 
0 & 0
\end{pmatrix}
,\ P_{2} =
\begin{pmatrix}
0 & \mathbb{R} \\ 
0 & \mathbb{Q}
\end{pmatrix}
,\ \text{and }R.
\]
In particular we see that $R$ has finitely many ideals and thus satisfies
both ACC and DCC on ideals. In addition, $R$ is not prime since $I^{2}=0$,
and $\Spec(R) = \{P_1,P_2\}$ with $R/P_1 \cong \mathbb{Q}$ and
$R/P_2 \cong \mathbb{R}$. Notice that $I$ is the smallest nonzero
ideal of $R$ and that $R/I\cong \mathbb{R}\times \mathbb{Q}$
is semisimple, hence right artinian. Let $\F_{a}$ (resp.\ $\F_{n}$) 
denote the family of all ideals whose factor rings are right artinian
(resp. right noetherian). Then we see that every nonzero ideal belongs
to both $\F_a$ and $\F_{n}$, but the zero ideal is not prime. Then by
the Prime Ideal Principle neither $\F_{a}$ nor $\F_{n}$ can be an
Oka family. Additionally $R$ is left artinian but not right artinian or
right noetherian (see \cite[(1.22)]{FC}), even though it has
finitely many ideals and all of its prime quotients are right noetherian.

Following Artin, Small, and Zhang~\cite{ArtinSmallZhang}, we say that
an algebra $R$ over a commutative noetherian ring $C$ is \emph{strongly right
noetherian} if, for every commutative noetherian ring $D$ containing $C$
as a subring, the $D$-algebra $R \otimes_C D$ is right noetherian. (An
example of Resco and Small~\cite{RescoSmall} shows that a noetherian
algebra over a field need not be strongly noetherian. On the other hand, Bell has
shown~\cite{Bell} that this condition is automatically satisfied for a countably generated
right noetherian algebra over an uncountable algebraically closed field.) 
These algebras play an important role in noncommutative algebraic geometry,
especially in the theory of point modules~\cite{ArtinZhang}.

Similar to the above examples, the strong right noetherian property provides
another $r$-Oka family.

\begin{proposition}
Let $C$ be a commutative noetherian ring and let $R$ be a $C$-algebra.
If every principal ideal of $R$ is finitely generated as a right ideal, then 
the semifilter 
\[
\F = \{I \unlhd R : R/I \text{ is a strongly right noetherian $C$-algebra}\}
\]
is $r$-Oka. If $R$ right noetherian, then it is strongly right noetherian over $C$
if and only if $R/P$ is strongly right noetherian for all (minimal) primes $P \in \Spec(R)$.
\end{proposition}

\begin{proof}
Let $\C$ be the class of all right $R$-modules $M$ such that, for any commutative
noetherian overring $D \supseteq C$, the right $R \otimes_C D$-module $M \otimes_C D$
is noetherian. (For brevity, we write $M^D = M \otimes_C D$ below.)
The statement will directly follow from Proposition~\ref{right module categories}
if we show that $\C$ is closed under extensions and quotients.

Fix a commutative noetherian overring $D \supseteq C$. 
Given a short exact sequence $0 \to L \to M \to N \to 0$ of $R$-modules where
$L,N \in \C$, right exactness of the tensor product gives an exact sequence
$L^D \to M^D \to N^D \to 0$ of $R^D$-modules. Letting $K \subseteq M^D$ denote
the kernel of $M^D \to N^D$, we obtain a short exact sequence of $R^D$ modules:
\[
0 \to K \to M^D \to N^D \to 0.
\]
But $K$ Is a homomorphic image of the noetherian $R^D$-module $L^D$. So in the
above, $K$ and $N^D$ are noetherian, implying that $M^D$ is noetherian. A similar
argument using right exactness shows that $\C$ is also closed under quotient modules.
\end{proof}

All of the examples of Oka families encountered up to this point actually
satisfied stronger properties than the Oka condition. The next proposition
will allow us to construct some Oka families without relying on any of the
stronger conditions.

\begin{proposition}
\label{Tensor category}
Let $k$ be a commutative ring, and let $\C$ be a class of $k$-modules containing 
the zero module that is closed under quotients, extensions, and tensor products. 
Then for any $k$-algebra $R$, the family $\F = \{ I\unlhd R:R/I\in \C \}$ is an Oka
semifilter.
\end{proposition}

\begin{proof}
The family $\F$ is a semifilter because $\C$ is closed
under quotients. To see that $\F$ is Oka, let $I\unlhd R$ and $a\in R$
be such that $(I,a), (a)^{-1} I, I(a)^{-1} \in \F$. Consider the following 
short exact sequence of $(R,R)$-bimodules:
\[
0\rightarrow \frac{(I,a)}{I} \rightarrow \frac{R}{I}\rightarrow 
\frac{R}{(I,a)}\rightarrow 0.
\]
We want to conclude that $R/I\in \C$. Because $R/(I,a)
\in \C$ and $\C$ is closed under extensions, it suffices to show that
$(I,a)/I \in \C$. Consider the $k$-bilinear map
$R/I(a)^{-1} \times R/(a)^{-1} I\rightarrow (I,a)/I$
given by 
\[
(r+I(a)^{-1},\ s+(a)^{-1}I) \mapsto ras+I.
\]
This induces a $k$-module homomorphism $R/I(a)^{-1} \otimes_{k} R/(a)^{-1} I
\rightarrow (I,a)/I$ that is
evidently surjective. The fact that $\C$ is closed under tensor
products and quotients then implies that $(I,a)/I \in \C$ as desired.
\end{proof}

Part (2) of the following proposition gives a method of finding categories
$\C$ satisfying the hypothesis of Proposition \ref{Tensor category}.

\begin{lemma}
\label{creating tensor categories}
Let $k$ be a ring and let $\C$ be a class of $k$-modules containing the zero module.
\begin{enumerate}
\item If $\C$ is closed under quotients and extensions, then it is closed 
under finite sums of modules in the following sense: if $N_1,\ldots, N_m\in \C$ 
are submodules of a single $k$-module $M$, then $\sum N_i \in \C$ also. 
\item Assume that $k$ is commutative. If $\C$ is closed under
quotients and extensions and if $M,N\in \C$ with $N$ finitely generated, 
then $M\otimes _{k}N\in \C$. In particular, if every module in $\C$
is finitely generated, then $\C$ is closed under tensor products.
\end{enumerate}
\end{lemma}

\begin{proof}
(1) It suffices to prove the statement for a sum of two submodules,
for then an easy inductive argument extends the result to arbitrary finite sums. So
suppose $N_1, N_2$ are submodules of the $k$-module $M$. Then we have the
following short exact sequence:
\[
0\rightarrow N_1 \rightarrow N_1 + N_2 \rightarrow (N_1 + N_2)/N_1 \rightarrow 0.
\]
But $(N_1 + N_2)/N_1 \cong N_2/(N_1 \cap N_2) \in \C$ because $\C$ is closed under
quotients. Thus the short exact sequence above gives $N_1 + N_2 \in \C$ as desired.

(2) Let $x_{1}, \dots, x_{n}\in N$ form a generating set. Each of the submodules
$M\otimes x_{i}\subseteq M\otimes N$ is a homomorphic image of $M$ (under the
map $- \otimes x_i \colon M \to M \otimes x_i$) and thus
lies in $\C$. But then $M\otimes N=\sum_{i=1}^{n}M\otimes x_{i}\in\C$.
\end{proof}

Now we take a look at some specific examples.

\begin{example}
Let $R$ be an algebra over a commutative ring $k$. The family $\F$ of all ideals $I$ of $R$
such that $R/I$ is module-finite over $k$ is an Oka semifilter (using 
Proposition~\ref{Tensor category} applied to the class $\C$ of finitely generated $k$-modules). 
Thus, an ideal $I \unlhd R$ maximal with respect to having $R/I$ not module-finite over 
$k$ is prime.
\end{example}

The result above is related to the study of ``just infinite'' algebras. An algebra $R$
over a commutative ring $k$ is is said to be \emph{just infinite} if $R$ is not finitely 
generated as a $k$-module, but every proper factor of $R$ is module-finite over~$k$.  
A result attributed to Lance Small (see \cite[Exercise~6.2.5]{Rowen}, \cite[Prop.~3.2]{PassmanTemple},
or~\cite[Lem.~2.1]{FarkasSmall}) states that in case $k$ is a field, an infinite dimensional
$k$-algebra has a just infinite factor ring (so that the family above has the Zorn
property) which is prime. 

The ``maximal implies prime'' statement above recovers part of Small's result. (For
an alternative proof, see~\cite{FarinaPendergrassRice}.) To recover the stronger result
using the Zorn property, we require the following Artin-Tate-style lemma.
We follow the proof strategy of~\cite[Lemma~3.1]{PassmanTemple}. 
Recall that a $k$-algebra is \emph{affine} if it is finitely generated as a $k$-algebra.

\begin{lemma}
Let $k$ be a commutative noetherian ring, and let $R$ be an affine $k$-algebra. 
If $M$ is a finitely generated right $R$-module with a submodule $L \subseteq M$
such that $M/L$ is finitely generated as a $k$-module, then $L$ is also finitely 
generated as an $R$-module.
(In particular, if $I$ is a right ideal of $R$ such that $R/I$ finitely generated over~$k$, then 
$I$ is finitely generated as a right ideal.)
\end{lemma}

\begin{proof}
Let $\{x_1,\dots,x_m\}$ be $k$-algebra generating set of $R$ and let $\{g_1, \dots, g_n\}$
be a set of $R$-module generators of $M$. Because $M/L$ is finitely generated over $k$, there 
exist $v_1,\dots,v_t \in M$ such that, for $V = \sum kv_i$, we have $M = L + V$. 

Consequently there exist elements $p_\alpha \in L$ and $q_{\beta \gamma} \in L$ such that 
each $g_\alpha \in p_\alpha + V$ and each $v_\beta x_\gamma \in q_{\beta \gamma} + V$. 
Let $L_0 \subseteq L$ be the $R$-submodule generated by the finite set of the $p_\alpha$ 
and $q_{\beta \gamma}$. It is straightforward to see that $L_0 + V$ is an
$R$-submodule of $M$. By construction, each $g_\alpha \in L_0 + V$, so $M = L_0 + V$.

Finally, if $l \in L$ then $l = l_0 + v$ for $l_0 \in L_0 \subseteq L$ and $v \in V$. But also
$v = l - l_0 \in L$. It follows that $L = L_0 + (L \cap V)$. Because $k$ is noetherian and $V$
is finitely generated over $k$, the $k$-submodule $L \cap V$ is also finitely generated 
over $k$. Thus $L = L_0 + (L \cap V)$ is finitely generated as an $R$-module.
\end{proof}

\begin{theorem}
Let $R$ be an affine algebra over a commutative noetherian ring $k$. The family
\[
\F = \{I \unlhd R : R/I \mbox{ is finitely generated as a $k$-module}\}
\]
satisfies $(P_1)$. Thus any just infinite $k$-algebra is prime. Furthermore, an algebra $R$ 
that is not finitely generated as a $k$-module is just infinite over $k$ if and only if $R/P$ is 
module-finite over~$k$ for all nonzero $P \in \Spec(R)$.
\end{theorem}

\begin{proof}
The previous lemma implies that every ideal in $\F$ is finitely generated as a right
ideal. Applying Proposition~\ref{right module categories}(2) to the class $\C$ of right
$R$-modules that are finitely generated over $k$, we conclude that $\F$ is a $(P_1)$
family. The final two claims respectively follow from the Prime Ideal Principle~\ref{PIP}
and Theorem~\ref{PIP supplement}.
\end{proof}

We present two more examples of Oka families derived from Proposition~\ref{Tensor category}.
In both examples, $k$ denotes a commutative ring.

\begin{example}
The class of noetherian $k$-modules is closed under quotients and
extensions, and consists of finitely generated modules. Then as above, the family of all
$I\unlhd R$ with $R/I$ noetherian over $k$ is an Oka semifilter.
\end{example}

\begin{example}
The class of $k$-modules of finite (composition) length, which is the
same as the category of finitely generated artinian $k$-modules, is closed under quotients
and extensions and consists of finitely generated modules. So the family of $I\unlhd R$ such
that $R/I$ is artinian over $k$ is an Oka semifilter.
\end{example}

We conclude with some thoughts on future directions in the study of the ``maximal implies
prime'' phenomenon for two-sided ideals in noncommutative rings. When comparing the
theories of Oka families for commutative rings~\cite{LR} and for right ideals in noncommuative
rings~\cite{Reyes} with the theory developed above, we notice one striking difference.
In both of the former two settings, it was shown that Oka families in a ring $R$ are in bijective 
correspondence with classes of cyclic $R$-modules that are \emph{closed under extensions} 
(see~\cite[\S 4]{LR} and~\cite[\S 4]{Reyes}).
This allows one to take any module-theoretic property that is preserved by extensions and produce
a corresponding Oka family, a method which led to a number of novel examples of ``maximal
implies prime'' results in both cases. 

By contrast, we have not found any (bi)module-theoretic characterization of the Oka families
of Definition~\ref{def:Oka} leading to a correspondence between Oka families in a ring $R$ 
and certain classes of (bi)modules. While Propositions~\ref{right module categories} 
and~\ref{Tensor category} provide sufficient conditions for a class of bimodules to give rise
to an Oka family, both of these methods seem severely limited by the fact that they can
only produce Oka \emph{semifilters} of ideals. (Note that each of those results involved
a short exact sequence of $(R,R)$-bimodules the form
\[
0 \to \frac{(I,a)}{I} \to \frac{R}{I} \to \frac{R}{(I,a)},
\]
where $I$ is an ideal and $a$ is an element of the ring $R$. In each case, the hypotheses 
are imposed with the intent of finding a description of $(I,a)/I$ in terms of the bimodules
$R/I(a)^{-1}$ and $R/(a)^{-1}I$. The key noncommutative difficulty seems to be the
lack of such a description in full generality.)

Thus, in conclusion, we offer two general problems to stimulate further work in this direction.
\begin{enumerate}
\item Is it possible to characterize Oka families of ideals in a noncommutative ring $R$ in terms of
suitable classes of $(R,R)$-bimodules (perhaps using a more appropriate definition than the one
given in Definition~\ref{def:Oka})?
\item In lieu of such a characterization, is there a related method that yields more examples of Oka 
families that are not semifilters and that do not satisfy conditions $(P_1)$--$(P_3)$? 
\end{enumerate}
Even partial progress on these problems seems likely to lead to new ``maximal implies prime'' results
for ideals in noncommutative rings.

\bibliographystyle{amsplain}
\bibliography{twosidedpip-arxiv-v3}
\end{document}